\documentclass[11pt, reqno]{amsart}

\usepackage{ amssymb, amsmath, amsthm,  graphicx, psfrag}
\usepackage[usenames,dvipsnames]{color}
\usepackage{enumerate} % for \bn[1.]
\usepackage{algpseudocode}

\usepackage{graphicx}
\usepackage[colorlinks=true, allcolors=blue]{hyperref}

\usepackage[utf8]{inputenc}
\usepackage{amsmath, amsthm, amssymb}
\usepackage{mathtools}
\usepackage[dvipsnames]{xcolor}
\usepackage{enumerate}

\theoremstyle{plain}
\newtheorem{theorem}{Theorem}
\newtheorem{corollary}[theorem]{Corollary}

\newtheorem{proposition}[theorem]{Proposition}

\theoremstyle{definition}
\newtheorem{definition}[theorem]{Definition}
\newtheorem{example}[theorem]{Example}
\newtheorem{note}[theorem]{Note}
\newtheorem{remark}[theorem]{Remark}

\title{Dual Frame Completion Problem}

\author{Roza Aceska}
\address{Department of Mathematical Sciences, Ball State University,   IN}
\email{raceska@bsu.edu} 

\author{Yeon Hyang Kim}
\address{Department of Mathematics, Central Michigan University, MI}
\email{kim4y@cmich.edu}

\author{Sivaram K. Narayan}
\address{Department of Mathematics, Central Michigan University, MI}
\email{naray1sk@cmich.edu}

\date{2024.}

\begin{document}
\maketitle

\begin{abstract}
In this paper we present the construction of an exact dual frame under specific structural assumptions posed on the dual frame.  When given a  frame $F$ for a finite dimensional Hilbert space, and a set of vectors $H$  that is assumed to be a subset of a dual frame of $F$, we answer the following question:   Which dual frame $G$ for $F$ - if it exists - completes the given set $H$?  Solutions are explored through a direct and an indirect approach, as well as via the singular value decomposition of the synthesis operator of $F$.

\end{abstract}

 %
 
 % \LARGE
 
 %\newpage

 \section{Introduction }
  
Frames are a generalization of bases and share some properties with bases. The redundancy of vectors in a (finite) frame allows for protection against loss of information during transmission of data. 
This redundancy is not present in bases and makes frames more adaptable for ensuring data integrity.
Moreover, frames allow for various specific properties, enabling the recovery of a signal through different reconstruction formulas that are not available for bases.
As an example, a collection of vectors $\{  x_i \}_{i=1}^k$
satisfying the reconstruction formula 
$x =\sum_{i=1}^k \langle x, x_i\rangle x_i$ is called a Parserval frame. 
In a smilar fashion, given a frame  $\{  x_i \}_{i=1}^k$ in  a Hilbert space $ \mathbb{F}^n$, 
a sequence $\{  y_i \}_{i=1}^k$ in $ \mathbb{F}^n$ is called a dual frame for $\{  x_i \}_{i=1}^k$ if it satisfies the reconstruction formula 
$x =\sum_{i=1}^k \langle x, y_i\rangle x_i$ for all $x$ in $ \mathbb{F}^n$, \cite{Ch2003, CMKLT06, CCNST,HKLW07}.

The problem of signal recovery from a partially erased data set of analysis coefficients
 was studied in  \cite{L15}, where the authors reconstruct an analyzed signal $f$ via a procedure they called bridging.   They replace the erased coefficients of   $f$ with respect to a frame $F$ by appropriate linear combinations of the non-erased coefficients, and they deliver a perfect reconstruction of $f$.  Implicitly, the bridging method is related to constructing a dual frame for $F$, which delivers perfect reconstruction for the signal at hand. The results presented in our paper here are more general, and serve the purpose of reconstruction of any signal, as our dual frame completion is exact, and often offers multiple dual frame sets. 
 
We explore the possibility of creating an exact dual frame under certain assumption for its structure, and with that provide a solution to the erasure problem eg. when the analysis coefficients are erased due to the malfunction of a portion of the measurement sensors, modeled by  the dual frame.  
More specifically, in this paper we answer the following dual frame completion question:  
Given a finite frame in $ \mathbb{F}^n$ denoted
 $F=\{f_1, f_2,  \hdots, f_k\}$, let 
 $\{h_1, h_2,  \hdots, h_s\}$ be a given set of $s$ vectors in $\mathbb{C}^n$ where $s \le k$. 
 Does there exist a dual frame $G =\{g_1, g_2,  \hdots, g_k\}$ of $F$ such that 
  $g_1=h_1, g_2=h_2, \ldots, g_s=h_s$?
  If such dual frame exists, the we say that  $G$ is a   dual frame completion of $\{h_1,  h_2, \hdots, h_s \}$. We may also say that the  dual frame completion problem  has a solution.

 \section{Preliminaries}

In this paper we study the dual frame completion problem. We assume that select vectors are in the dual frame of a fixed frame, and we explore the possibility of finding such a dual frame for the  fixed frame.  In section~\ref{sectionCompletion} we offer three approaches towards dual frame completion. %, while in the conclusion we offer ideas for future research in this direction.  
 We would like to state that our results  hold for all cases up to permutation.

  %; most definitions and results in this section can be found in \cite{undergradFrames}.
%By $ \mathbb{F}^n$  we denote a  finite-dimensional Hilbert space.  % $*$ instead of $T$.}

%%%%

\begin{definition} \label{framedefinition} Let $k\geq n$. 
A sequence of vectors $F=\{f_i\}_{i=1}^k$ in  $\mathbb{F}^n$ is a \textit{frame} for $\mathbb{F}^n$ if   there exist  real constants $0 < A \leq B < +\infty$ such that for every $f \in \mathbb{F}^n$, the following holds:
\begin{equation} A\lVert f \rVert ^2 \leq \sum_{i=1}^k|\langle f,f_i\rangle|^2 \leq B \lVert f \rVert ^2.\end{equation} 
  
%A frame is called \textit{tight} if $A = B$. 
\end{definition}

In Definition \ref{framedefinition},  $A$ is called the lower frame bound for $F$, and $B$ is called the upper frame bound for $F$. A frame    is called \textit{tight} if   $A=B$.  

\begin{note} 
Unless otherwise noted, in this paper we work with frames for $\mathbb{F}^n$ with $k$ vectors, $k \geq n$. By $F$ we denote both the frame $F=\{f_i\}_{i=1}^k$ for $\mathbb{F}^n$, and the $n\times k$ matrix $F=[f_1 \; \dots \; f_k]$.  \end{note} 

 In a  finite-dimensional space, frames are simply  spanning sets of the space.
All bases are frames by definition, but frames in general are not bases.  
Each frame has an associated  analysis, synthesis, and frame operator. 
 In signal processing, $F^*$ is  used to compute the analysis coefficients, while $F$ would be  used for synthesis, to reconstruct the signal. 
 
The \textit{frame operator} of  $F = \{f_i\}_{i=1}^k$ is  
 a linear operator on  $\mathbb{F}^n$, with matrix $S=FF^*$. 
It follows that the frame operator is a positive, self-adjoint, invertible operator.  
  
\begin{definition}
Let $\{f_i\}_{i=1}^k$ be a frame for $\mathbb{F}^n$. A \textit{dual frame} for $\{f_i\}_{i=1}^k$ is a frame $\{g_i\}_{i=1}^k$ such that for every $f \in \mathbb{F}^n$, \begin{equation*}
    f = \sum_{i=1}^k\langle f,g_i \rangle f_i = \sum_{i=1}^k\langle f,f_i \rangle g_i.
\end{equation*}%$$f = \sum_{i=1}^k\langle f,g_i \rangle f_i.$$
\end{definition}
 Thus,  
a frame $G = \{g_i\}_{i=1}^k$  is a dual frame for $F$ if  (and only if)   $FG^*=GF^*=I_{n\times n}$.
In this case the {\it analysis} is performed with one frame, and the {\it synthesis} with the other frame.

 The dual frame is not necessarily unique; in general, a frame can have infinitely many dual frames. 
Every frame does have at least one dual frame, though; this is the \textit{canonical dual frame}. 
The  {canonical dual frame} for a frame $F = \{f_i\}_{i=1}^k$ for $\mathbb{F}^n$ with frame operator $S$ is the frame $\tilde{F} = \{S^{-1}f_i\}_{i=1}^k$.  
The frame operator for $\tilde{F}$ is $S^{-1}$.
  
With  tight frames, a useful feature is  that its frame operator $S$ %of a tight frame
is a scalar multiple of the identity operator, where that scalar is the frame bound $A=B$. Thus, the canonical dual frame   of a tight frame $F = \{f_i\}_{i=1}^k$,  is simply  
$\{\frac{1}{A}f_i\}_{i=1}^k$. When a frame $F$ is tight, the representation of a vector $f$ in $\mathbb{H}$ depends only on $F$, % using only $f$, the vectors in $F$, and the constant $a$ 
that is, we only need $F$ for the analysis and synthesis.

Since $F$ is a frame, it has full row rank. And the dual frame completion involves finding a dual frame $G$
such that $FG^*=I$, 
we use the following theorem to address the dual frame completion problem.
\begin{theorem} \label{f1}
Let $A$, $B$ and $C$ be some matrices such that $AB=C$.
The following statements are equivalent:
\begin{enumerate}
\item The system of linear equations \( AB = C \) has a solution.
\item \( \text{rank}(A) = \text{rank}([A \; C]) \).
\item The columns of \( C \) are in the span of the columns of \( A \).
\end{enumerate}
If \( A \) has full row rank, then
   the pseudoinverse \( A^+ \) is given by
$$ A^+ = A^* (A A^*)^{-1},$$
and a particular solution \( B \) can be found using \( A^+ \) as
   $B = A^+ C.$
\end{theorem}

 \section{A surgery problem in finding the dual frame }\label{sectionCompletion}
 
In this section we  find all the possible dual frames $G$ of a known frame $F=\{f_1, \hdots, f_k\}$ for  $\mathbb{F}^n$, $k\geq n$, when the first few 
vectors in $G$ are fixed.  
Given a set of vectors $h_1, h_2, \hdots, h_s \in \mathbb{F}^n$, $s\leq k$, if  there exists a dual frame $G =\{g_1, \hdots, g_k\}_{i=1}^k$ for $F$ such that $g_1=h_1, g_2=h_2, \ldots, g_s=h_s$, then we say that $G$ is a {\it  dual frame completion} of $h_1,  h_2, \hdots, h_s$ i.e. we say that the problem of dual frame completion has a solution.  
 Depending on the application, specific characteristics of the dual frame may be prioritized. For example, one might aim to retain certain directional vectors in the dual frame, or alternatively, minimize the number of nonzero entries within each vector.

To address this problem, we use  Theorem \ref{f1} to classify the conditions under which a dual frame completion is possible. When a solution exists, the pseudo-inverse method is employed to construct a specific solution, ensuring practical implementation of the desired completion. 

\subsection{On the direct approach}

A dual frame pair $(F, G)$ for $\mathbb{F}^n$ satisfies the equality $FG^*=I_{n\times n}$. In this subsection we utilize  said  equality  to solve the problem of dual frame completion. This is the  {\it direct approach} to the dual frame completion problem.
We begin with the following example.
 
\begin{example}\label{exwithkernels}  Let
%\begin{center}
$\displaystyle F =[F_0 \; F_1]$,   where $\displaystyle F_0=\begin{bmatrix}
      1 &2  \\1&1 
  \end{bmatrix}$, $F_1=\begin{bmatrix}
       1 &-1&1\\ 0&1&2
  \end{bmatrix}$. %\end{center}
 The columns of matrix  $F$ form  a frame in $\mathbb{R}^2$. It shows that  the columns of the matrix $G$, where 
\begin{center} $  G=\begin{bmatrix}
 -1 & 1 & -3x & -2x & x\\2&-1 &-3y & -2y & y
 \end{bmatrix}=[G_0 \; G_1]$, with $G_0=\begin{bmatrix}
     - 1 &1 \\2&-1 
  \end{bmatrix}$, 
\end{center} form a  dual frame for $F$, for all  $x, y \in \mathbb{R}$.
\end{example}

Observe that in Example \ref{exwithkernels},   $ F_0 $   is a basis for $\mathbb{R}^2$, with dual basis
$G_0$, so $ F_0G_0^* =I$. 
Furthermore, by the assumption that $G$ is a dual frame of $F$ we have $I=FG^*= F_0G_0^* +F_1G_1^*$, thus $ F_1 G_1^T = 0, $ which is equiavlent to 
$(\text{ker}G_1)^\perp \subseteq \text{ker} F_1.$ 
The above observation leads us the following result.

 %%  
%\section{Dual basis completion}
 %%  

  %\vspace{8.5mm}

  \begin{theorem}\label{thmKernel}
  Let $F_0$ be a frame for $\mathbb{F}^n$, with dual frame  $G_0$, where $|F_0|=|G_0|=s\geq n$. Let  $F=[  F_0 \; F_1]_{n\times k}$ be a  frame for $\mathbb{F}^n$.
Given any $G_1 \in \mathbb{F}^{n\times (k-s)}$,  we have that 
%\begin{center}
\begin{equation}\label{dualbasF}
G=[G_0 \; G_1] \; \text{  is a dual frame for $F$} \end{equation} if and only if % whenever
\begin{equation}\label{kernelCond}(\text{ker}G_1)^\perp \subseteq \text{ker} F_1.\end{equation}  

 In particular, if the columns of $F_1$ are linearly independent, then the only possible choice of $G$ for the completion of the dual frame is $G=[G_0 \;0]$.
     \end{theorem}

  \begin{proof}
As demonstrated in  the discussion following  Example \ref{exwithkernels}  , \( G = [G_0 \; G_1] \) is a dual frame for \( F \) if and only if 
\((\ker G_1)^\perp \subseteq \ker F_1. \)
By matrix product properties, if the columns of \( F_1 \) are linearly independent, then \( F_1 G_1^* = 0 \) if and only if \( G_1 = 0 \).
\end{proof}

 Theorem~\ref{thmKernel} is particularly  useful in applications when there is a need to expand the dual frame pair, in order  to increase the number of  samples so that the increased redundancy of the analysis frame would provide better protection against data loss. Starting with a dual frame pair  $(F_0, G_0)$, one can  expand frame $F_0$ to a larger frame $F$, and respectively expand $G_0$ to a dual frame for $F$. We see that if the expansion from frame $F_0$ to $F$ is preconditioned i.e. we want to enrich $F_0$ with the columns of $F_1$ then we have exactly the answer when we can expand the dual frame $G_0$ of $F_0$ to a dual frame for $F$.  
By Theorem~\ref{thmKernel}, when $F_1$ has linearly independent columns, the dual frame expansion is trivial, that is $G_1=O$. We can have a nontrivial expansion if we scale the vectors in $G_0$.

 \begin{corollary}\label{anotherC}  
   Let   $F=[F_0 \;\; F_1]$ be a frame for $\mathbb{F}^n$, where $F_0=F_{n\times s}$ is a frame, $F_1= F_{n\times (k-s)}$,   
 and let $G_0 $ be  a dual frame for $F_0$, $ 1 \leq s \le  k$.  Let   $w_1, \hdots, w_s \in \mathbb{F}$ be  some nonnegative scaling coefficients and let $W=diag\{w_1, \hdots, w_{s}\}$.
 If $F_1$ has  linearly independent columns, then 
 $G=[G_0W \; G_1]$ form a dual frame pair if and only if 
$$G_1=(F_1^*F_1)^{-1} \left( I_n - G_0WF_0^*\right) F_1.$$
\end{corollary}  
\begin{proof} We note that $G=[G_0W \; G_1]$ form a dual frame pair if and only if 
$ G_1F_1^*= I_{n\times n} - G_0WF_0^*$.
If $F_1$ has  linearly independent columns,  since $F_1^*F_1$ is invertible, 
$$G_1=(F_1^*F_1)^{-1} \left( I_n - G_0WF_0^*\right) F_1.$$
\end{proof}
 We note that if $F_1 = \{f \}$, then  $F_1^*F_1 = \| f\|^2 $ and if  $|F_1| = n$, then 
 we get $ G_1 = (F_1^*)^{-1}(I_{n\times n} - G_0WF_0^*)$.

%%%%%%%%%%%%%%
%%%%%%%%%%%%
%%%

In Theorem~\ref{thmKernel}, the analysis began with the assumption that a dual frame for a subframe of $F$ was provided. Building upon this foundation, the following theorem addresses the dual frame completion problem in a more general context.

Here, we consider an arbitrary set of given vectors 
$h_1, \hdots, h_s$. Using a direct approach, we establish conditions for the existence of a dual frame completion that satisfies these constraints, providing a systematic method to construct such a dual frame when possible. This generalization broadens the applicability of dual frame theory to scenarios where predefined vector constraints are integral to the solution.

\begin{theorem} \label{direct} 
 Let $F=\{f_1 , \hdots , f_k\}$  be a frame for $\mathbb{F}^n$ with $k>n$. %with $|F|=k>n$   with $|F|=k>n$
 Let $H=[h_1  \, \hdots \, h_s]$, $ 1 \leq s \le  k-n$ be a given $n \times s$ matrix.
 
 \begin{itemize} 

\item[i.] Let  
$F_{n \times s}=[f_1 \, \hdots \, f_s]$. There exists a dual frame   
$$G = \{h_1, \hdots, h_s , g_{s+1} , \hdots , g_{k}\} $$  for $F$  if and only if 
  the columns of   $I -F_{n \times s} H^*$ are in the span of the  columns of $F_{n \times (k-s)} =[f_{s+1} \, \hdots \, f_k]$. In the special case, when $s=k-n$, the completion is unique    if and only if  $f_{k-n+1},\hdots, f_k$  form a basis for $\mathbb{F}^n$.% and  $F_{n \times (k-s)}=\{f_{s+1}, \hdots, f_k\}$. %  are the submatrices  of the first $s$ and the last  $k-s$ columns, respectively. 

\item[ii.] When $s< k-n$,  if  $\{f_{s+1},\hdots, f_k\}$  is a frame for $\mathbb{F}^n$, then there exist infinitely many  dual frame completions. %  a   dual frame completion $g_{s+1}, \hdots, g_k$ for selected $\{ h_1, \hdots, h_s \}$ exists  

% \item[iii.] When $s=k-n$, the completion is unique    if and only if  $f_{k-n+1},\hdots, f_k$  form a basis for $\mathbb{F}^n$. {\color{red} compare to Example~\ref{wghtDirect}.}

  \end{itemize}
 
 \end{theorem}

\begin{proof}
i.  Let $G = [H \; G_{n \times (k-s)}]$  be a dual frame for $F=[ F_{n \times s}  \;F_{n \times (k-s)} ]$, where 
$G_{n \times (k-s)}= [g_{s+1} \, \hdots \, g_k]$. 
Then from $FG^* = I$, we have 
\begin{equation} \label{eqdirect} F_{n \times (k-s)}  G_{n \times (k-s)}^* = I -F_{n \times s} H^*. \end{equation}
Thus,  by Theorem \ref{f1}, there exists a dual frame   $G = \{h_1, \hdots, h_s, g_{s+1}, \hdots, g_{k} \} $  for $F$  if and only if 
  the columns of   $I -F_{n \times s} H^*$ are in the span of columns of $F_{n \times (k-s)}$. 

When $s=k-n$, the result follows from \eqref{eqdirect} since the matrix  $ F_{n \times (k-s)} $     is invertible    if and only if  $f_{s+1},\hdots, f_k$  form a basis for $\mathbb{F}^n$.

ii.  Assume  $f_{s+1},\hdots, f_k$  form a { frame}  for $\mathbb{F}^n$. Then   the columns of   $I -F_{n \times s} H^*$ are always in the span of $F_{n \times (k-s)}$. By Theorem~\ref{direct} (i). and    \eqref{eqdirect}, there exist infinitely many ways to perform the dual frame  completion.  

%iii.  When $s=k-n$ and when   $f_{k-n+1},\hdots, f_k$  form a basis for $\mathbb{F}^n$, the matrix $ F_{n \times (k-(k-n))}=[f_{s+1}\,\hdots \, f_k]$     is invertible, and by  \eqref{eqdirect}. 
\end{proof}
\begin{remark}
If  $F_{n \times (k-s)} $ has full row rank in the above theorem, we obtain $G_{n \times (k-s)}$ using the pseudoinverse formula from Theorem~\ref{f1}, 
which is  the minimum norm solution for the equation \eqref{eqdirect}. 
We note that if 
$H=0$, then the pseudoinverse formula provides the canonical dual frame of $F$. 
\end{remark}
%The canonical dual frame is stable under small perturbations in the frame:
% \begin{theorem}\cite{Ch2003}
%Let \( \{f_i\}_{i \in I} \) be a frame with frame operator \( S \) and frame bounds \( A \) and \( B \). If the perturbed frame \( \{f_i'\}_{i \in I} \) satisfies:
%\[
%\sum_{i \in I} \|f_i - f_i'\|^2 \leq \delta^2 < \frac{A}{2},
%\]
%then \( \{f_i'\}_{i \in I} \) is also a frame with bounds close to those of the original frame. The canonical dual frames \( \{g_i\}_{i \in I} \) and  \( \{g_i'\}_{i \in I} \) of \( \{f_i\}_{i \in I} \)  and \( \{f_i'\}_{i \in I} \) satisfies:
%\[
%\|g_i - g_i'\| \leq C \delta,
%\]
%where \( C \) is a constant depending on the frame bounds of the original frame.
% \end{theorem}

\begin{example} 
Let $F$ be a frame for $\mathbb{R}^2$ with $|F|=k>2$. Let $s=1$, and let $h_1 \in \mathbb{R}^2 \setminus \{0\}$ be an  element of a dual frame $G=\{h_1, g_2, \hdots, \}$   of $F$.   
The equality  $ FG^T = I_{2\times 2}$ implies 
\begin{equation}  F_{2 \times (k-1)}  G_{2 \times (k-1)}^* = I -F_{n \times s} h_1^*. \end{equation}
Thus a   dual frame $\mathcal{G}$ for $\mathcal{F}$ with $h_1 \in \mathcal{G}$  does not exist exists  if 
$$rank ( F_{2 \times (k-1)}) \neq
   rank ( [F_{2 \times (k-1)}  \; I -F_{n \times s} h_1^*]).
   $$
When $k=3$,  by Theorem \ref{direct} (iii),  if $f_2$ and $f_3$ are not collinear there is a unique dual frame $\mathcal{G}$ for $\mathcal{F}$. 
When the frame vectors $f_2$ and $f_3 $ are collinear, the problem has infinitely many  solutions
if $rank ( F_{2 \times (k-1)}) =
   rank ( [F_{2 \times (k-1)}  \; I -F_{n \times s} h_1^*]).
   $    \end{example}

 \begin{example}\label{wghtDirect}
Let $F_1=\begin{bmatrix} 1&1 \\1&-1 \end{bmatrix}$,   $F_2=\begin{bmatrix} 1&0\\ 0&1 \end{bmatrix}$, $F= [F_1\; F_2]$, and  $H=\begin{bmatrix} .5 & .5\\ .5 & -.5 \end{bmatrix}$.

\vspace{2.3mm}
 
 Observe that $F$ is a frame for $\mathbb{R}^2$, while  $F_1$ is a basis for   $\mathbb{R}^2$, with  dual basis  $H$. By Theorem~\ref{thmKernel}, the only dual frame of $F$ type $G=[H \; G']$  is the trivial one: 
\[G=\begin{bmatrix} .5 & .5 &0&0\\ .5 & -.5 &0&0 \end{bmatrix}. \] 
By introducing scaling  of the  vectors in $H$ (see \cite{CDKNSZ,CKLMNPS14, DKN15}), we find a dual frame  for $F$:
\[G_{x,y}=\begin{bmatrix} .5x & .5y &1-\frac{x+y}{2}&0\\ .5x & -.5y &0&1-\frac{x+y}{2} \end{bmatrix}, \;  \text{ where $x$, $y$ are any real numbers.}  \]
\end{example}

 We explore further the reasons why in  Example~\ref{wghtDirect} we can find a nontrivial scaled dual frame.   Let $1\leq s\leq k$, and let   $F_1=F_{n\times s}$, $F_2= F_{n\times (k-s)}$ be some matrices such that the columns of $F=[F_1 \; F_2]$ form  a frame for $\mathbb{F}^n$,      and fix  $H=[h_1 \hdots h_{s}]$. 
  Suppose  a nontrivial dual frame completion is not possible  as obseved  in Example~\ref{wghtDirect}, or we simply want greater freedom in constructing a dual frame. Assume that some scaling coefficients $w_1, \hdots, w_s$ exist such that the columns of $I-F_1(HW)^*$ are   in the span of $F_2$, where $W=diag\{w_1, \hdots, w_{s}\}$. Thus, for some $G_1 \in \mathbb{F}^{n\times (k-s)}$ we  have \begin{equation} \label{eqdirectWghtCond} F_2  G_1^* = I -F_1W^* H^*. \end{equation}

 Hence, we have the following corollary, which generalizes Corollary \ref{anotherC}.  

 \begin{corollary}\label{CDAWspan}  Let  the columns of  $F=[F_1 \;\; F_2]$ form  a frame for $\mathbb{F}^n$, where $F_1=F_{n\times s}$, $F_2= F_{n\times (k-s)}$,      and let $H=[h_1 \hdots h_{s}]$, $1\leq s\leq k$. Let $w_1, \hdots, w_s \in \mathbb{F}\setminus \{0\}$ and set $W=diag\{w_1, \hdots, w_{s}\}$.   There exists some  $G_1=[g_{s+1} \; \hdots \; g_k] \in \mathbb{F}^{n\times (k-s)}$ such that $G=[HW \;\; G_1]$ is a dual frame for $F$ if and only if  the columns of $   I -F_1W^* H^*$ are in the span of $ F_2 $.  (that is, if and only if \eqref{eqdirectWghtCond} holds true for some $G_1$).  \end{corollary}

\begin{proof}
We want the frame equality to  hold,  that is  $F_0W^*H^* + F_1G_1^*=I$, which is equivalent to equation \eqref{eqdirectWghtCond}. Equivalently, such a $G_1$ exists if and only if the columns of $I - F_0W^*H^*$ are in the span of matrix $F_1$.  
\end{proof}

\subsection{On the indirect approach - via product matrices}

Let $F$ be a frame for $\mathbb{F}^n$.  % the following procedure can be used to find a dual Let frame $G$  of $F$ (as discussed in [FU, p 160]).  
Let $0$ denote the $  (k-n)\times n$ zero matrix, and let  $A$ be  any 
$n\times (k-n)$ matrix.
 If $P$ is 
the product of   matrices corresponding to the row operations  such that 
\begin{equation}\label{conditionP}
PF^* = 
\begin{bmatrix}
I_n \\ 0
\end{bmatrix}, \end{equation}    then   $ \left( [I_n \; A] \, P \right) F^*= [I_n \; A] \,  \left( PF^*\right)=I_n +0=I_n.$ Thus %any dual frame  $G$ of $F$ is given by 
\begin{equation}\label{dualGP}
G= [I_n \; A] \, P \end{equation}  is a dual frame for $F$  (see \cite{HKLW07}, pp.160).  %, $G$ is a  dual frame of $F$.
 % \begin{remark} Matrix $P$ is simply a product of elementary matrices, thus it is   invertible. \end{remark}

%%% Proposed version of 
 In this subsection, we will use the submatrix notation for a given matrix $A$ of size $k \times k$, given by
  \[ A=\begin{bmatrix}
A_{n \times s} & A_{n \times (k-s)} \\
A_{(k-n)\times s} & A_{(k-n) \times (k-s)} \\
\end{bmatrix} ,\]
 where 
    $A_{a \times b}$  are corresponding  submatrices of size    $a \times b$.

\begin{theorem}\label{resultIAP}

 Let $F$ be a frame for $\mathbb{F}^n$ with $|F|=k>n$.  
 Let $H=[h_1  \, \hdots \, h_s],  1 \leq s \le  k-n$ be a given $n \times s$ matrix. 
  Let 
  \[ P=\begin{bmatrix}
P_{n \times s} & P_{n \times (k-s)} \\
P_{(k-n)\times s} & P_{(k-n) \times (k-s)} \\
\end{bmatrix}  \; \text{ satisfy
 \eqref{conditionP}.}\]
  Then

\begin{itemize} 

\item[i.]  There exists a dual frame   $G = \{h_1, \hdots, h_s, g_{s+1}, \hdots, g_{k} \} $  for $F$  if and only if 
  the columns of $H^*-P_{n\times s}^*$ are in the span of columns of  $P_{(k-n)\times s}^*$.

 \item[ii.] If in addition $s= k-n$ and $P_{(k-n)\times (k-n)}$ is an  invertible matrix, then   the dual frame completion is unique.  

  \end{itemize}
 
\end{theorem}
 
 %%%%%%%%%%%%%%%%%
\begin{proof}
We seek for a dual frame $G = \{h_1, \hdots, h_s, g_{s+1}, \hdots, g_{k} \}$.
We consider equality \eqref{dualGP}
%$$  \begin{bmatrix} I_n & A \end{bmatrix} \, P = G $$
  for an unknown  
$n\times (k-n)$ matrix $A$.
Let $G'=\{g_{s+1}, \hdots, g_{k} \}.$ 
Then  
$$
 \begin{bmatrix}
I_n & A
\end{bmatrix} \, 
\begin{bmatrix}
P_{n \times s} & P_{n \times (k-s)} \\
P_{(k-n)\times s} & P_{(k-n) \times (k-s)} \\
\end{bmatrix}
= \begin{bmatrix}
H & G'
\end{bmatrix}.$$
 Thus we have 
$ P_{n \times s}+ A  P_{(k-n)\times s} =H$. Taking conjugate transpose on both sides, we have 
\begin{eqnarray}\label{thm3eq} 
 P_{(k-n)\times s}^* A^*  =H^*- P_{n \times s}^*.
\end{eqnarray}
So, by Theorem \ref{f1}, if  the columns of $H^*-P_{n\times s}^*$  are in the span of columns of   $P_{(k-n)\times s}^*$, then we obtain $A^*$ from 
\eqref{thm3eq}, which provides a dual frame $G$ from \eqref{dualGP}. 
If $G$ is a dual frame of $F$, then from \eqref{dualGP}, there exists $A$, which satisfies \eqref{thm3eq}. 
Therefore, 
 the columns of $H^*-P_{n\times s}^*$  are in the span of columns of   $P_{(k-n)\times s}^*$ if and only if $G$ is a dual frame for $F$.

Whenever  $P_{(k-n)\times (k-n)}$ is invertible, it holds 
\begin{eqnarray}\label{matrixA}
A =(H-I_nP_{n \times (k-n)}) P_{(k-n)\times (k-n)}^{-1}.
\end{eqnarray}
This implies that 
\begin{eqnarray}\label{Gprime}
G'=  P_{n\times n} + A P_{(k-n) \times n},
\end{eqnarray} which provides a dual frame completion of $F$. 
Therefore, we conclude that  $[H\; G']$ is a dual frame of $F$.
\end{proof}
The following examples illustrate various cases of Theorem~\ref{resultIAP}.
 \begin{example}
The columns of  $F=\begin{bmatrix} 1&0&-1&-2\\0&1&-2&-4\end{bmatrix}$ form a frame for $\mathbb{R}^2$ and the respective product matrix is \[P=\begin{bmatrix} 1&0&0&0\\0&1&0&0\\ 1&2&1&0\\2&4&0&1 \end{bmatrix}.\] %Clearly the lower-left corner $2\times 2$ submatrix of $P$ is singular. 
If $H= \begin{bmatrix} 1&3\\2&4\end{bmatrix}$, then since 
Theorem \ref{f1} (2) does not hold, there is no dual frame completion of type $G=[H\; G']$.
 \end{example} 

 \begin{example}
The columns of  $F=\begin{bmatrix} 1&0&-1&-2 &1\\0&1&0&4&-2\end{bmatrix}$ form a frame for $\mathbb{R}^2$ and the respective product matrix is \[P=\begin{bmatrix} 1&0&0&0&0\\0&1&0&0&0\\ 1&0&1&0&0\\2&-4&0&1&0\\-1&2&0&0&1 \end{bmatrix}.\]   If $H= \begin{bmatrix} 1&2&1 \\3&4&\frac{9}{2}\end{bmatrix}$, then since 
Theorem \ref{f1} (2) is true,  there is a dual frame completion of type $G=[H\; G']$. That is,  the columns of any matrix $G= \begin{bmatrix} 1&2&1&\frac{a-1}{2}&a\\3&4&\frac{9}{2}&\frac{2b-3}{4}&b\end{bmatrix}$ will  form a dual frame for $F$. 

 \end{example}

\begin{remark}\label{remkngrt}    In   case $s>k-n$, we sometimes have a chance to solve the problem of frame completion: we temporarily  discard vectors $h_{k-n+1}, \hdots, h_s$ and  solve the problem for $h_1, \hdots, h_{k-n}$. Then, we  check if the vectors $h_{k-n+1}, \hdots, h_s$ are equal to the first $s-k+n$ columns of $P_{n\times n} + A P_{(k-n) \times n}$ (or not).  In the affirmative case, we have a solution to the posed problem; otherwise there is no solution. \end{remark}

\begin{example}\label{exampleGenCoef}  As observed in Remark~\ref{remkngrt} when $s>k-n$,  the submatrix $G'$ in (\ref{Gprime}) produces a dual frame if 
$h_{k-n+1}, \ldots, h_s$ are the first $s-k+n$ columns  of $G'$.
For example,  let 
\[F= \begin{bmatrix}
1 & 0 & 0 & 2\\
0 & 1 & 0  & 0\\
0 & 0  & 1 & 0
\end{bmatrix} \text{ and } G = \begin{bmatrix}
2 & 0 & x & a \\
1 & 1 & y & b\\
3 & 0 & z & c
\end{bmatrix}  \; \text{  which makes $s=2>k-n=1$} .\]
Since $F$ is very sparse, we need only one basic row operation to apply to $F^T$, that is to multiply $F^T$ by the invertible matrix $P$
\begin{equation}\label{ex9P}
P = \begin{bmatrix}
1& 0 & 0 & 0\\
0 & 1 & 0 & 0 \\
0 & 0 & 1 & 0\\
-2 & 0 & 0 & 1
\end{bmatrix}.
\end{equation}
%Assuming $a_4 \neq 0$ (otherwise, matrix $P$ in \eqref{ex9P} would be singular), we relax the requirement of known vectors in $G$. 

Let
$H= [2  \; 1 \; 3]^T$.  Then from equations \eqref{matrixA} and \eqref{Gprime}, we have 
$$ G' = \begin{bmatrix}
 0 & 0 & -1/2 \\
 1 & 0 & -1/2\\
 0 & 1 & -3/2
\end{bmatrix}.  $$
Since the second column of $G$ agrees with the first column of $G'$, we find 
$G= \begin{bmatrix}
H & G'
\end{bmatrix} $ 
as the  desired dual frame. 

\vspace{2.3mm}

\end{example}

 In the following, we consider the dual completion problem  when $$rank( P_{(k-n)\times s}^*) =0.$$
 \begin{corollary}\label{Rcs1} 
Let $F$ be a frame for $\mathbb{F}^n$, with {  $|F|=k\geq n+1$}. Suppose $P$ is an invertible $k\times k$ matrix such that \eqref{conditionP} holds true. \\
 i. Let  $h_1 \in \mathbb{F}^n$. 
If $rank( P_{(k-n)\times 1}^*) =0$, then $G=[ P_{n\times 1}  \; G']$ is a dual frame for some $G'$.\\
 ii. Let $H=[h_1  \, \hdots \, h_s],  1 \leq s \le  k-n$ be a given $n \times s$ matrix.
  If $H = P_{n\times s}$, and $P_{(k-n)\times s} =0$, then $G=[ P_{n\times s}  \; G']$ is a dual frame for some $G'$.
\end{corollary}
\begin{proof}
Assume $F$ is a frame for $\mathbb{F}^n$,  {  $|F|=k\geq n+1$}, and let $P$ be an invertible $k\times k$ matrix satisfying \eqref{conditionP}.   
If $rank( P_{(k-n)\times s}^*) =0$, then  by equation  \eqref{thm3eq} and  equation {\it 2} in  Theorem \ref{f1}, both cases are true  when $P_{n\times s}=H$.
\end{proof}

\begin{remark}
Theorem \ref{resultIAP} will hold true even if some of the selected $h_1, \hdots, h_s$  are zero vectors.
In the event that one or more of these selected vectors, for instance, $h_1$, is a zero vector, it is possible to remove the first vector of frame $F$, as ilustrated in  Example~\ref{ex0dual}.  
\end{remark}

 \begin{example}\label{ex0dual} Consider the frame 
  $ F = \begin{bmatrix}
1& 2 & 3 & 4\\
4 & 3 & 2 & 1
\end{bmatrix} $   for $\mathbb{R}^2$. Any of its dual frames can be described as 
\[G = \begin{bmatrix}
a&c&-3a-2c-1/5&2a+c+2/5\\b&d&-3b-2d+4/5&2b+d-3/5
\end{bmatrix}, \; \text{for any real numbers $a,b,c,d$}.\]
%
% Copy Paste this in Chrome
%   https://www.wolframalpha.com/input?i=Solve%5B%7B%7B1%2C2%2C3%2C4%7D%2C%7B4%2C3%2C2%2C1%7D%7D.%7B%7Ba%2Cb%7D%2C%7Bc%2Cd%7D%2C%7B-3a-2c-1%2F5%2C+-3b-2d%2B4%2F5%7D%2C%7B2a%2Bc%2B2%2F5%2C2b%2Bd-3%2F5%7D%7D%3D%3D%7B%7B1%2C0%7D%2C%7B0%2C1%7D%7D%2C+a%2Cb%2Cc%2Cd%5D
%
If we fix the choice of the first vector $(a, b)^T$ in $G$, we can find infinitely many such dual frames $G$. 
 The problem for dual frame completion for fixing the first two columns of $G$ is solvable since the respective product matrix $P$ (satisfying \eqref{conditionP}) has an invertible lower-left corner submatrix $  P_{(4-2)\times (4-2)}$:
\[P=\frac{1}{5}\begin{bmatrix}-3&4&0&0\\2&-1&0&0\\5&-10&5&0\\10&-15&0&5 \end{bmatrix}, \; P_{(4-2)\times (4-2)}=\frac{1}{5}\begin{bmatrix} 5&-10\\10&-15  \end{bmatrix}.\]

Choosing the first column to be the zero column, there are infinitely may dual frames of $F$ with the following structure:
\[G = \begin{bmatrix}
0&c&-2c-1/5&c+2/5\\0&d&4/5 -2d&d-3/5
\end{bmatrix}.\]

If we remove the first column of $F$ to obtain 
 $ F' = \begin{bmatrix}
2 & 3 & 4\\
 3 & 2 & 1
\end{bmatrix} $
then the dual frame of $F'$ is 
$ G' = \begin{bmatrix}
c&-2c-1/5&c+2/5\\
d&4/5 -2d&d-3/5
\end{bmatrix}$.

%In addition, if all $a,b,c,d=0$, we still have  that  $FG^T=I_{2\times 2}$; that is, we can remove the first two vectors of $F$ and the resulting set is a frame (basis)  for $\mathbb{R}^2$. 
 \end{example} 
 
 The above example illustrates the concept of  frame surgery, which is studied in \cite{CKLMNS13}.

\vspace{2.3mm}

%\begin{remark} Corollary~\ref{Rcs1} answers the question if a dual frame completion is possible under specific assumptions  but it is not delivering the structure of the desired dual frame. Provided that such a dual frame exists,  its content is delivered via \eqref{dualGP}.\end{remark}

Scaling coefficients give us more freedom to complete a dual frame, as discussed in the direct approach section. 
  
\begin{proposition}\label{wghtProp}
Let $F$ be a frame for $\mathbb{F}^n$, with $|F|=k>n$.  
 Let $P$ be a $k\times k$ product matrix such that \eqref{conditionP} is satisfied. In particular, $P=[P_1 \; P_2]$, where $P_1$ is the left $k\times s$ submatrix of $P$, and $P_2$ is the right $k\times (k-s)$ submatrix of $P$.

Suppose the choice of   $H=\{h_1,\hdots, h_{s} \}\subset \mathbb{F}^n$ is fixed for some $s\leq k$; consider the set $G=\{h_1,\hdots, h_s, g_{s+1},\hdots, g_{k}\}$  for some unknown $g_{s+1},\hdots, g_{k} \in \mathbb{F}^n$. 

Let $w_1, \hdots, w_s \in \mathbb{F} \setminus\{0\}$,   let  $W=diag\{w_1, \hdots, w_s, 1,\hdots, 1\}$ be    a diagonal $k\times k$ matrix, and let $W_H= diag \{w_1, \hdots, w_s\}$ be a  diagonal $s\times s$ matrix. Then:

\vspace{2.13mm}

  $GW$ is a dual frame for $F$   if and only if there exists a $ n\times (k-n)$ matrix $A$ such that 
\begin{equation} \label{submatrixA} H= [I_n \; A] P_1 W_H^{-1}.\end{equation}
\end{proposition}

%%%

\begin{proof}
Let $GW$ be a dual frame for $F$ for some choice of $w_1, \hdots, w_s \in \mathbb{F} \setminus\{0\}$ and  $g_{s+1},\hdots, g_{k} \in \mathbb{F}^n$.  This is equivalent to the existence of some $n\times (k-n)$ matrix  $A$ such that $GW=[I_n \; A]P$. In block-matrix representation we have:
\[GW=[H_{n\times s} \; G'_{n\times (k-s)}]W=[HW_H \;\; G'], \; \text{while} \;  [I_n \; A] P = [I_n \; A] [P_1 \; P_2].\] We have: $$[HW_H \; G'] = [I_n \; A] [P_1 \; P_2],$$ that is, $HW_H=[I_n \; A] P_1$, which is equivalent to \eqref{submatrixA}.

Note that once we have found $A$, we also find $G'=[I_n \; A]   P_2$. Also note that  once we have determined the scaling coefficients $w_1, \hdots, w_s$, then we have matrix $A$ as a function of the scaling coefficients: 
\[ A=\begin{bmatrix} O_{n\times ( k-n)}\\I_{k-n}\end{bmatrix} GWP^{-1} .\]

\end{proof}

\subsection{The Singular Value Decomposition Approach } % and Dual Frame Completion}
The  singular values decomposition (SVD) of a frame can be used 
 to characterize all of its dual frames, as  described here: 
Given a frame $F \in \mathbb{F}^{n\times k}$ for $\mathbb{F}^{k}$, we write its SVD as a matrix product, $F=U\Sigma V^*$. The matrices  $U  \in \mathbb{F}^{n\times n}$ and $V \in \mathbb{F}^{k\times k}$ are unitary,  while 
$\Sigma  \in \mathbb{F}^{n\times k}$ is a diagonal matrix whose diagonal entries are  the singular values of $F$, namely  $\sigma_1, \sigma_2, \hdots, \sigma_n$, which are positive, and arranged in a non-increasing order. 
Suppose $G$ is  a frame for $ \mathbb{F}^{k}$, and let %In \cite{GittaSVD}, 
$M_G := U^*GV$. It shows  \cite{GittaSVD} that $G$ is a dual frame of $F$ if and only if there exists a matrix  $X \in \mathbb{F}^{n\times (k-n)}$ such that 
\begin{equation}\label{MofG} \; 
M_G= [\Sigma^{-1} \; X ]\in \mathbb{F}^{n\times k},
\end{equation}
where $\Sigma^{-1} $ is a diagonal matrix whose diagonal entries are  $1/\sigma_1, 1/\sigma_2, \hdots, 1/\sigma_n$.
Thus, a frame $G$ is a dual frame of of $F$ if and only if it  can be  written as the product
\begin{equation}\label{SVDdual}
G=UM_GV^* .
\end{equation}
The canonical dual frame of $F$ is
obtained by placing a $n\times (k-n)$ zero matrix as matrix $X$   in \eqref{MofG}.  

In this subsection we use  characterization \eqref{SVDdual}  to solve the dual frame completion problem.

\begin{theorem}\label{SVDresult} 
Let $F$ be a frame for $\mathbb{F}^n$, $|F|=k>n$, and let $H=[h_1  \, \hdots \, h_s],  1 \leq s \le  k$ be a given $n \times s$ matrix. Let $F=U\Sigma V^*$ be the singular value decomposition of $F$. 
Denote 
 \[ V^*=\begin{bmatrix}
V^*_{n \times s} & V^*_{n \times (k-s)} \\
V^*_{(k-n)\times s} & V^*_{(k-n) \times (k-s)} \\
\end{bmatrix} .\]
\begin{itemize} 
\item[i.] 

 There exists a dual frame   $G = \{h_1, \hdots, h_s, g_{s+1}, \hdots, g_{k} \} $  for $F$  if and only if 
  the columns of $H^*U - (V_{n\times s}^*)^*(\Sigma^{-1})^*$ are in the span of columns of  $(V^*_{(k-n)\times s})^*$.

 \item[ii.] If in addition $s= k-n$ and $V^*_{(k-n)\times (k-n)}$ is an  invertible matrix, then   the dual frame completion is unique.  

\item[iii.] The vectors $h_1, \hdots, h_s$ are the first $s$ vectors in the canonical dual of $F$ if and only if 
\begin{equation}
h_j = U\begin{bmatrix} v_{j,1}^*/\sigma 1  \\\vdots  \\v_{j,n}^*/\sigma_n  \end{bmatrix} \;  \text{ for all } j \in \{1, \hdots, s\}.
\end{equation}
  \end{itemize}
 
\end{theorem}

\begin{proof}
Let $F$ be a frame for $\mathbb{F}^n$. Suppose $h_1\hdots, h_s \in \mathbb{F}^n$ are fixed. We aim to find a dual frame $G$ of $F$  such that $G =\{h_1 \hdots, h_s,g_{s+1}, \hdots, g_k\}$ for some unknown $g_{s+1}\hdots, g_k \in \mathbb{F}$. We must have $G=UM_GV^* $, that is,
\begin{equation}\label{fromSVD}
U^*[h_1 \; \hdots \;  h_s \;  g_{s+1} \;  \hdots \;  g_k]=U^*G=M_GV^*
\end{equation} 
 for some unknown $X$  in \eqref{MofG}. 
 This implies that 
 \begin{equation}\label{keySVD}
  X V^*_{(k-n)\times s} = U^*H -\Sigma^{-1} V^*_{n\times s}
  \end{equation} 
  Then by Theorem \ref{f1}, we have the result (i).\\
For ($ii$), let $s=k-n$ and assume that the submatrix $V^*_{(k-n)\times s}$ is   invertible. Then we have 
\(X=\left(V^*_{(k-n)\times s}\right)^{-1}(U^*H -\Sigma^{-1} V^*_{n\times s}),\)
which defines a unique $n\times k$ matrix $M_G$ using \eqref{MofG}. 
 To prove ($iii$) note that the columns of $H$ are the first $s$ vectors of the canonical dual frame of $F$ if and only if the matrix $X$ in \eqref{MofG} is the zero matrix. Now the conclusion follows from \eqref{keySVD}.
\end{proof}

%\begin{example} The SVD for frame $F$ in Example\ref{exampleGenCoef} is \[F=U\Sigma_FV^T = \begin{bmatrix} 1&0&0\\0&0&1\\0&1&0\end{bmatrix} \begin{bmatrix} \sqrt{5}&0&0&0\\0&1&0&0\\0&0&1&0\end{bmatrix} \begin{bmatrix} \frac{1}{\sqrt{5}}&0&0&\frac{2}{\sqrt{5}}\\0&0&1&0\\0&1&0&0\\\frac{-2}{\sqrt{5}}&0&0&\frac{1}{\sqrt{5}}\end{bmatrix}.  \] \end{example}

%\vspace{2.3mm}

%
%If $P$ is 
%the product of   matrices corresponding to the row operations  such that 
%\begin{equation}\label{conditionP}
%PF^* = 
%\begin{bmatrix}
%I_n \\ O
%\end{bmatrix}, \end{equation} where $O$ is a $  (k-n)\times n$ zero matrix, then $G$ is given by 
%\begin{equation}\label{dualGP}
%G= [I_n \; A] \, P \end{equation}  for any 
%$n\times (k-n)$ matrix $A$. %, $G$ is a  dual frame of $F$.

\begin{remark}

We note that there is relation between the characterization of a dual frame 
$G_1$ using  the  singular values of the frame $F$ and 
$G_2$ using product   matrix $P$ corresponding to the row operations as in
equations \eqref{conditionP} and 
 \eqref{dualGP}. 
 To see this, we rewrite  $G_1=UM_GV^*$ as 
 $G_1= U \Sigma^{-1} \begin{bmatrix}
I_n & B
\end{bmatrix} V^*$, where 
 \begin{equation}
B = \begin{bmatrix}  x_{1,1}/\sigma_1 & \hdots & x_{1,k-n}/\sigma_1 \\
 x_{2,1}/\sigma_2&\hdots & x_{2,k-n}/\sigma_2\\
\vdots&\hdots & \vdots\\
 x_{n,1}/\sigma_n &\hdots & x_{n,k-n}/\sigma_n \end{bmatrix}. 
\end{equation}
Now, let 
$
G_2= \begin{bmatrix}
I_n & A
\end{bmatrix} \, P$. 
When $A=B$,  we have 
$$  G_2 = \Sigma U^{-1} G_1 (V^*)^{-1} P.$$ 
 \end{remark}

\begin{example}
The SVD of the frame studied in Example~\ref{exampleGenCoef} is 
\[F=U\cdot \Sigma_F \cdot V^T=\begin{bmatrix} 1&0&0\\0&0&1\\0&1&0\end{bmatrix}\cdot \begin{bmatrix} 5^{1/2}&0&0&0\\0&1&0&0\\0&0&1&0\end{bmatrix} \cdot  \left(5^{-1/2}\begin{bmatrix}  1&0&0& 2  \\0&0&1&0\\0&1&0&0\\-2  &0&0&   1\end{bmatrix}\right). \] By equation \eqref{SVDdual},  any dual frame $D$ of $F$ equals
\[D= U\cdot \begin{bmatrix} 5^{-1/2}&0&0&s\\0&1&0&t\\0&0&1&w\end{bmatrix} \cdot V^T=\frac{1}{5} \begin{bmatrix} 1-2\cdot 5^{1/2}s & 0&0& 5^{1/2} s +2\\ -2\cdot 5^{1/2}w & 5&0& 5^{1/2} w\\-2\cdot 5^{1/2}t & 0&5& 5^{1/2} t\end{bmatrix} \] 
for any $  s, t, w\in \mathbb{R}$.
The dual frame $G$ that was initially purused in Example~\ref{exampleGenCoef} can be obtained via the SVD approach for the choice $s=w=- {1}/{2\sqrt{5}}$, $t=-{3}/{2\sqrt{5}}$. We can easily see that equation \eqref{keySVD} is satified, and by Theorem~\ref{SVDresult}(ii.), since $s=k-n=4-2=2$, the desired dual frame must be unique as the determinant of the $2\times 2$ submatrix in $V^T$ in the lower-left corner is $2/\sqrt{5}$ (so the submatrix  is invertible). 
\end{example}

\end{document}